\documentclass[fontsize=11pt,paper=a4,reqno,numbers=noendperiod,bibliography=totoc]{scrartcl}
\usepackage[greek,USenglish]{babel}
\usepackage[T1]{fontenc}
\usepackage[latin10]{inputenc}
\usepackage{amsmath,amsthm} 
\usepackage{fouriernc}
\usepackage{dsfont,upgreek}
\usepackage[obeyspaces]{url}
\usepackage{colortbl,color,xcolor}
\usepackage{graphicx,tikz,pgfplots}
\usepackage{centernot}
\usepackage{array}
\usepackage{enumitem}
\usepackage[bookmarks=true,plainpages=false,linktocpage,colorlinks=true,citecolor=green!80!black,linkcolor=red!70!black,filecolor=magenta,urlcolor=magenta,breaklinks,pdfauthor={Thomas Jahn},pdftitle={Successive Radii and Ball Operators in Generalized Minkowski Spaces}]{hyperref}
\usepackage[babel]{microtype}
\usepackage{subfigure}
\usepackage{scrhack}
\usepackage{float}
\usepackage{cite}

\usetikzlibrary{calc,intersections,through,backgrounds,arrows,patterns,shapes.geometric,shapes.misc}

\newcommand{\KK}{\mathcal{K}}
\newcommand{\LL}{\mathcal{L}}

\newcommand{\RR}{\mathds{R}}
\newcommand{\fall}{\:\forall\:}
\newcommand{\ex}{\:\exists\:}
\newcommand{\norm}[1]{\left\lvert#1\right\rvert}         
\newcommand{\mnorm}[1]{\left\lVert#1\right\rVert}        
\newcommand{\setn}[1]{\left\{#1\right\}}                 
\newcommand{\setcond}[2]{\left\{#1 \:\middle\vert\: #2\right\}}    
\newcommand{\setconds}[2]{\{#1 \:\vert\: #2\}} 
\newcommand{\defeq}{\mathrel{\mathop:}=}

\newcommand{\dah}{i.e., }
\newcommand{\lr}[1]{\!\left(#1\right)}

\newcommand{\cRR}{\overline{\mathds{R}}}
\newcommand{\strline}[2]{\left\langle#1,#2\right\rangle}
\newcommand{\clseg}[2]{\left[#1,#2\right]}
\newcommand{\skpr}[2]{\left\langle#1 \,\middle\vert\, #2\right\rangle}
\newcommand{\enquote}[1]{``#1''}
\newcommand{\norel}{\mathrel{\phantom{=}}}
\newcommand{\noequiv}{\mathrel{\phantom{\Longleftrightarrow}}}

\theoremstyle{plain}
\newtheorem{Satz}{Theorem}[section]
\newtheorem{Kor}[Satz]{Corollary}
\newtheorem{Lem}[Satz]{Lemma}
\newtheorem{Prop}[Satz]{Proposition}

\theoremstyle{definition}
\newtheorem{Def}[Satz]{Definition}

\newtheorem{Bem}[Satz]{Remark}

\DeclareMathOperator{\Proj}{proj}

\DeclareMathOperator{\co}{co}
\DeclareMathOperator{\cl}{cl}
\DeclareMathOperator{\inte}{int}

\DeclareMathOperator{\dist}{dist}
\DeclareMathOperator{\bi}{bi}
\DeclareMathOperator{\bh}{bh}
\DeclareMathOperator{\cc}{cc}
\DeclareMathOperator{\ic}{ic}
\let \eps \varepsilon

\let \subset \subseteq
\let \supset \supseteq

\addtokomafont{caption}{\small}
\setkomafont{captionlabel}{\bfseries}

\begin{document}
\parindent 0pt
\title{Successive Radii and Ball Operators in Generalized Minkowski Spaces}
\author{Thomas Jahn\\
{\small Faculty of Mathematics, Technische Universit\"at Chemnitz}\\
{\small 09107 Chemnitz, Germany}\\
{\small thomas.jahn\raisebox{-1.5pt}{@}mathematik.tu-chemnitz.de}
}

\date{}
\maketitle
\allowdisplaybreaks[2]

\begin{abstract}
We investigate elementary properties of successive radii in generalized Minkowski spaces (that is, with respect to gauges), \dah we measure the \enquote{size} of a given convex set in a finite-dimensional real vector space with respect to another convex set. This is done via formulating some kind of minimal containment problems, where intersections or Minkowski sums of the latter set and affine flats of a certain dimension are incorporated. Since this is strongly related to minimax location problems and to the notions of diametrical completeness and constant width, we also have a look at ball intersections and ball hulls.
\end{abstract}



\section{Introduction}
Modifying the geometric configuration of the classical minimax location problem, one might ask for an affine flat which approximates a given set in a minimax sense. Along with the description of the \emph{size} of the given set via sections of the unit ball, this is the setting for the theory of \emph{successive radii}.
Following the footsteps of classical convex geometry, this theory includes generalizations of the famous inequalities of Jung \cite{Jung1901,Henk1992,BoltyanskiMa2006} and Steinhagen \cite{Steinhagen1922,BetkeHe1993a} (which are upper bounds for the circumradius in terms of the diameter, and for the minimum width in terms of the inradius, respectively).
But it has also connections to the notions of mixed volume and lattices \cite{HenkHe2008b,HenkHe2009} as well as the behaviour under Minkowski addition of input sets \cite{GonzalezHe2012,Gonzalez2013}. Besides the Euclidean theory, some knowledge about successive radii was also gained in normed spaces \cite{BarontiCaPa1997,BarontiPa1988,BarontiPa1994,GritzmannKl1992,GritzmannKl1993}.

But the metrical description of a set with respect to a fixed \enquote{unit ball} is more far-reaching than just assigning several numbers. The notions of constant width and diametrical maximality (also known as completeness) still form an active field of research within the theory of normed spaces. It is well-known that these notions are intimately related to special intersections of translates of the unit ball which appear under several names in the literature, such as \emph{wide spherical hull} and \emph{tight spherical hull} \cite{MorenoSc2012a,MorenoSc2012c}, or \emph{ball intersection} and \emph{ball hull} \cite{MartinMaSp2014,MartiniRiSp2013}.

In the present paper, we will extend the notions of successive radii, ball hulls, and ball intersections from normed spaces to generalized Minkowski spaces whose unit balls are not necessarily centered at the origin (but still convex). For Minkowski functionals, which serve as the corresponding analogues of norms, also the notions of \emph{gauges} or \emph{convex distance functions} are common.

The purpose of the present paper is twofold. In Section~\ref{chap:successive}, we investigate elementary properties of successive radii in generalized Minkowski spaces. This combines the relaxation of the geometric configuration with the relaxation of distance measurement, where the latter was already initiated in \cite{Jahn2015b}. Section~\ref{chap:translates} is devoted to ball hulls and ball intersections in generalized Minkowski spaces. Our presentation starts with recalling our notation in Section~\ref{chap:preliminaries}, and it ends with a discussion of possible future research in Section~\ref{chap:conclusions}.

\section{Preliminaries}\label{chap:preliminaries}
Throughout this paper, we shall be concerned with the vector space $\RR^d$, with the topology generated by the usual inner product $\skpr{\cdot}{\cdot}$ and the norm $\norm{\cdot}=\sqrt{\skpr{\cdot}{\cdot}}$ or, equivalently, its unit ball $B$. For the \emph{extended real line} we write $\cRR\defeq\RR\cup\setn{+\infty,-\infty}$, with the conventions $0(+\infty)\defeq +\infty$, $0(-\infty)\defeq 0$ and $(+\infty)+(-\infty)\defeq +\infty$. We use the notation $\KK^d_0$ for the class of convex compact sets having non-empty interior. The \emph{line segment} between $x$ and $y$ and the \emph{straight line} through $x$ and $y$ shall be denoted by $\clseg{x}{y}$ and $\strline{x}{y}$, respectively. The abbreviations $\dim$, $\inte$, $\cl$, and $\co$ stand for \emph{dimension}, \emph{interior}, \emph{closure}, and \emph{convex hull}, respectively. A set $K$ is said to be \emph{centrally symmetric} if there is a point $z\in\RR^d$ such that $K=2z-K$, and it is \emph{centered} iff $K=-K$.
The \emph{circumradius} and the \emph{inradius} of $K\subset\RR^d$ with respect to a set $C\in\KK^d_0$ are defined as
\begin{equation*}
R(K,C)=\inf_{x\in\RR^d}\inf\setcond{\lambda>0}{K\subset x+\lambda C}
\end{equation*}
and
\begin{equation*}
r(K,C)=\sup_{x\in\RR^d}\sup\setcond{\lambda\geq 0}{x+\lambda C \subset K},
\end{equation*}
respectively. For the sake of simplicity, we always assume $0\in\inte(C)$. Then, using the Minkowski functional $\gamma_C$ of $C$ which is defined via $\gamma_C(x)=\inf\setcond{\lambda >0}{x\in \lambda C}$, we can write
\begin{align}
\inf\setcond{\lambda>0}{K\subset x+\lambda C}=\sup_{y\in K}\gamma_C(y-x).\label{eq:inclusion_vs_optimization}
\end{align}

\section{Successive radii}\label{chap:successive}
The study of successive radii in Euclidean space, which goes back to the late 1970s \cite{Pukhov1979}, deals with circumradii and inradii (with respect to $C=B$) of sections of $K$ by, and projections of $K$ onto, affine subspaces of specified dimension. This research direction has been continued in \cite{Perelman1987,Henk1992,BetkeHe1992,BetkeHe1993a}. For normed spaces, Gritzmann and Klee \cite{GritzmannKl1992,GritzmannKl1993} already realized that it makes things easier if one replaces the procedure of projecting $K$ onto linear subspaces of dimension $j$ by considering cylindrical unit balls, \dah the Minkowski sum of $C$ and a linear subspace of dimension $d-j$.
There is no difference between these points of view in the Euclidean setting because sections of $B$ by, and projections of $B$ onto, parallel affine flats \enquote{all look the same}, \dah they are homothets of each other. Therefore, the intimate relationship between Euclidean distance measurement and the usual topology carries over to affine flats. Finding the appropriate definitions for $C\neq B$ is therefore sometimes difficult, see Remark~\ref{bem:projected_balls}. In the last six years, relationships between successive radii, Brunn--Minkowski theory, and discrete geometry were established; see \cite{HenkHe2008b,HenkHe2009}. But recent research \cite{GonzalezHe2012,Gonzalez2013} also ties in with the work of the early 1990s by introducing new quantities.
This is done via interchanging infima and suprema (over a family of linear subspaces), circumradii and inradii, and sections and cylinders. In the present paper, the symbols for these quantities are chosen in a way that clarifies which choices were made and what dimension the participating linear subspaces have. In all cases, we cite several papers which contain the classical definitions or the corresponding results within the classical setting.

We start with the most basic notions which are called \emph{Kolmogorov} and \emph{Bernstein diameters} in \cite{Pukhov1979}, \emph{external} and \emph{internal radii} in \cite{Perelman1987}, and \emph{inner} and \emph{outer $j$-radii} in \cite{GritzmannKl1992,GritzmannKl1993}. From now on, we generally assume that $K\in\RR^d$, $C\in\KK^d_0$, and $0\in\inte(C)$.
\begin{Def}\label{def:inner_outer}
Let $j\in\setn{1,\ldots,d}$, and define the \emph{distance function of $A$ with respect to $C$} via $\dist_C(\cdot,A):\RR^d\to\RR$, $\dist_C(y,A)\defeq \inf\setcond{\gamma_C(y-z)}{z\in A}$. Furthermore, define
\begin{enumerate}[label={(\alph*)},align=left]
\item{
$\begin{aligned}[t]R^\pi_j(K,C)&\defeq \inf\setcond{\lambda \geq 0}{\ex L\in \LL^d_{d-j} \ex x\in \RR^d: K\subset x+L+\lambda C}\\
&=\inf_{L\in\LL^d_{d-j}}\inf_{x\in\RR^d} \inf\setcond{\lambda \geq 0}{K\subset x+L+\lambda C}\\
&=\inf_{L\in\LL^d_{d-j}}\inf_{x\in\RR^d} \sup_{y\in K}\dist_C(y,x+L)\\
&=\inf_{L\in\LL^d_{d-j}}R(K,C+L),
\end{aligned}$
}
\item{
$\begin{aligned}[t]r_\sigma^j(K,C)&\defeq\sup_{L\in \LL^d_j}\inf_{x\in\RR^d}\sup_{y\in\RR^d}\sup\setcond{\lambda \geq 0}{y+\lambda(C\cap (x+L))\subset K}\\
&=\sup_{L\in \LL^d_j}\inf_{x\in\RR^d}r(K,C\cap(x+L)).
\end{aligned}$
}
\end{enumerate}
\end{Def}

\begin{Bem}\label{bem:projected_balls}
For $C=B$, we even have
\begin{equation*}
R^\pi_j(K,C)=\inf_{L\in \LL_j^d} R(\Proj_L(K),C)=\inf_{L\in \LL_j^d} R(\Proj_L(K),\Proj_L(C)),
\end{equation*}
see \cite[Definition~3.1, Remark~3.2]{BrandenbergKoe2013}. For $C\neq B$, outer radii of projections onto linear subspaces are not the same as outer radii of projections with respect to projections of $C$. Here the \emph{projections of $x\in\RR^d$ and $K$ onto $L$} are denoted by
\begin{align*}
\Proj_L(x)\defeq \setcond{y\in\RR^n}{\norm{x-y}=\inf_{z\in L}\norm{x-z}}
\end{align*}
and $\Proj_L(K)=\bigcup_{x\in K}\Proj_L(x)$, respectively.

Now take $d=3$, $L=\setcond{(x_1,x_2,x_3)\in\RR^2}{x_3=0}$, and
\begin{align*}
C&=\setcond{(x_1+2\lambda-1,x_2+2\lambda-1,2\lambda-1)}{x_1,x_2\in\RR, x_1^2+x_2^2\leq 1, \lambda \in \clseg{0}{1}}\\
&=\clseg{(1,0,1)}{(-1,0,-1)}+\setconds{(x_1,x_2,0)\in\RR^3}{x_1^2+x_2^2\leq 1}.
\end{align*} 
\begin{figure}[h!]
\begin{center}
\begin{tikzpicture}[line cap=round,line join=round,>=stealth]
\pgfsetyvec{\pgfpoint{-0.5cm}{-0.5cm}}
\pgfsetxvec{\pgfpoint{1cm}{0cm}}
\pgfsetzvec{\pgfpoint{0cm}{1cm}} 
\draw plot[shift={(1,0,1)},domain=0:2*pi,smooth,variable=\t]({cos(\t r)},{sin(\t r)},0);
\draw [dashed] plot[shift={(0,0,0)},domain=0:2*pi,smooth,variable=\t]({cos(\t r)},{sin(\t r)},0);
\draw [dashed] plot[shift={(-1,0,-1)},domain=pi:2*pi,smooth,variable=\t]({cos(\t r)},{sin(\t r)},0);
\draw plot[shift={(-1,0,-1)},domain=0:pi,smooth,variable=\t]({cos(\t r)},{sin(\t r)},0);
\draw [color=black] plot[shift={(-1,0,0)},domain=0.5*pi:1.5*pi,smooth,variable=\t]({cos(\t r)},{sin(\t r)},0);
\draw [color=black] plot[shift={(1,0,0)},domain=1.5*pi:2.5*pi,smooth,variable=\t]({cos(\t r)},{sin(\t r)},0);
\draw [color=black] (-1,1,0)--(1,1,0);
\draw [color=black,dashed] (-1,-1,0)--(1,-1,0);
\draw (-2,0,-1)--(-1.5,0,-0.5);
\draw (0,0,-1)--(0.5,0,-0.5);
\draw [color=black,dashed] (-1.5,0,-0.5)--(-1,0,0);
\draw [color=black,dashed] (0.5,0,-0.5)--(1,0,0);
\draw (-1,0,0)--(0,0,1);
\draw (1,0,0)--(2,0,1);
\draw (2,0,0) node[right]{$K$};
\draw (2,0,1) node[right]{$C$};
\end{tikzpicture}
\end{center}\caption{Projections of balls.}
\end{figure}
If $K=\Proj_L(C)$, then $R(K,C)=2$ but $R(K,\Proj_L(C))=1$.
\end{Bem}
All the quantities listed below appear in \cite{Gonzalez2013,BetkeHe1992}. The paper \cite{HenkHe2008b} also contains $R_\pi^j$ and $r^\sigma_j$. The quantities $R_\pi^j$ and $R_\sigma^j$ are investigated in \cite{BrandenbergKoe2013}. Finally, $r_\pi^j$ appears in \cite{GonzalezHeHi2015}.
\begin{Def}\label{def:successive_radii}
Let $j\in\setn{1,\ldots,d}$. Define
\begin{enumerate}[label={(\alph*)},align=left]
\item{$\begin{aligned}[t]
R_\pi^j(K,C)&\defeq\sup_{L\in \LL^d_{d-j}} R(K,C+L)\\
&=\sup_{L\in \LL^d_{d-j}} \inf\setcond{\lambda\geq 0}{\ex x\in X: K\subset x+L+\lambda C}\\
&=\sup_{L\in \LL^d_{d-j}} \inf_{x\in\RR^d}\inf\setcond{\lambda\geq 0}{K\subset x+L+\lambda C}\\
&=\sup_{L\in\LL^d_{d-j}}\inf_{x\in\RR^d}\sup_{y\in K}\dist_C(y,x+L),
\end{aligned}$}
\item{$R_\sigma^j(K,C)\defeq \sup_{L\in \LL^d_j}\sup_{x\in\RR^d} R(K\cap (x+L),C)$,}
\item{$R^\sigma_j(K,C)\defeq \inf_{L\in \LL^d_j}\sup_{x\in\RR^d} R(K\cap (x+L),C)$,}
\item{$r_\pi^j(K,C)\defeq \sup_{L\in \LL^d_{d-j}} r(K+L,C)$,}
\item{$r^\pi_j(K,C)\defeq \inf_{L\in \LL^d_{d-j}} r(K+L,C)$,}
\item{$\begin{aligned}[t]r^\sigma_j(K,C)&\defeq \inf_{L\in \LL^d_j}\inf_{x\in\RR^d}\sup_{y\in\RR^d}\sup\setcond{\lambda \geq 0}{y+\lambda(C\cap (x+L))\subset K}\\
&=\inf_{L\in \LL^d_j}\inf_{x\in\RR^d}r(K,C\cap(x+L)).\end{aligned}$}
\end{enumerate}
\end{Def}
Next we establish the monotonicity of $R_\pi^j$ and $R^\pi_j$ with respect to the dimension index $j$. In classical settings, these results are stated in \cite{HenkHe2008b,Gonzalez2013}, the first one also in \cite{HenkHe2009,GonzalezHeHi2015,GonzalezHe2012,GritzmannKl1992,Tikhomirov1960,Perelman1987}, and the second one also in \cite{BetkeHe1992,BrandenbergKoe2013}.
\begin{Satz}
We have $R_\pi^1(K,C)\leq\ldots\leq R_\pi^d(K,C)$ and $R^\pi_1(K,C)\leq\ldots\leq R^\pi_d(K,C)$.
\end{Satz}
\begin{proof}
Fix $j \in\setn{1,\ldots,d-1}$. Let $L\in\LL^d_{d-j}$ and $L^\prime\in \LL^d_{d-j-1}$ be such that $L^\prime \subset L$. For all $x\in\RR^d$, we obviously have
\begin{equation*}
\inf\setcond{\lambda\geq 0}{K\subset x+L+\lambda C}\leq \inf\setcond{\lambda\geq 0}{K\subset x+L^\prime+\lambda C}.
\end{equation*}
Since $x\in\RR^d$ is arbitrary, we obtain
\begin{equation*}
\inf_{x\in\RR^d}\inf\setcond{\lambda\geq 0}{K\subset x+L+\lambda C}\leq \inf_{x\in\RR^d}\inf\setcond{\lambda\geq 0}{K\subset x+L^\prime+\lambda C}.
\end{equation*}
It follows that
\begin{align*}
\inf_{x\in\RR^d}\inf\setcond{\lambda\geq 0}{K\subset x+L+\lambda C}&\leq \inf_{\substack{L^\prime\in\LL^d_{d-j-1}\\L^\prime\subset L}} \inf_{x\in\RR^d}\inf\setcond{\lambda\geq 0}{K\subset x+L^\prime+\lambda C}\\
&\leq\sup_{\substack{L^\prime\in\LL^d_{d-j-1}\\L^\prime\subset L}} \inf_{x\in\RR^d}\inf\setcond{\lambda\geq 0}{K\subset x+L^\prime+\lambda C}.
\end{align*}
Consequently,
\begin{align*}
R_\pi^j(K,C)&=\sup_{L\in\LL^d_{d-j}}\inf_{x\in\RR^d}\inf\setcond{\lambda\geq 0}{K\subset x+L+\lambda C}\\
&\leq \sup_{L\in\LL^d_{d-j}}\inf_{\substack{L^\prime\in\LL^d_{d-j-1}\\L^\prime\subset L}}\inf_{x\in\RR^d}\inf\setcond{\lambda\geq 0}{K\subset x+L^\prime+\lambda C}\\
&\leq\sup_{L\in\LL^d_{d-j}}\sup_{\substack{L^\prime\in\LL^d_{d-j-1}\\L^\prime\subset L}} \inf_{x\in\RR^d}\inf\setcond{\lambda\geq 0}{K\subset x+L^\prime+\lambda C}\\
&=R_\pi^{j+1}(K,C)
\end{align*}
and
\begin{align*}
R^\pi_j(K,C)&=\inf_{L\in\LL^d_{d-j}}\inf_{x\in\RR^d}\inf\setcond{\lambda\geq 0}{K\subset x+L+\lambda C}\\
&\leq \inf_{L\in\LL^d_{d-j}}\inf_{\substack{L^\prime\in\LL^d_{d-j-1}\\L^\prime\subset L}}\inf_{x\in\RR^d}\inf\setcond{\lambda\geq 0}{K\subset x+L^\prime+\lambda C}\\
&=R^\pi_{j+1}(K,C).
\end{align*}
We obtain $R_\pi^j(K,C)\leq R_\pi^{j+1}(K,C)$ and $R^\pi_j(K,C)\leq R^\pi_{j+1}(K,C)$.
\end{proof}
In \cite{Gonzalez2013,BetkeHe1992}, we also find the monotonicity of $R^\sigma_j$ and $R_\sigma^j$ with respect to $j$.
\begin{Satz}
We have $R_\sigma^1(K,C)\leq \ldots\leq R_\sigma^d(K,C)$ and $R^\sigma_1(K,C)\leq\ldots\leq R^\sigma_d(K,C)$.
\end{Satz}
\begin{proof}
Fix $j \in\setn{1,\ldots,d-1}$. Let $L\in\LL^d_j$ and $L^\prime\in \LL^d_{j+1}$ be such that $L\subset L^\prime$. For all $x\in\RR^d$, we obviously have $K\cap (x+L)\subset K\cap (x+L^\prime)$, and hence
\begin{equation*}
\inf_{y\in\RR^d}\inf\setcond{\lambda\geq 0}{K\cap(x+L)\subset y+\lambda C}\leq \inf_{y\in\RR^d}\inf\setcond{\lambda\geq 0}{K\cap(x+L^\prime)\subset y+\lambda C}.
\end{equation*}
It follows that
\begin{align*}
\inf_{y\in\RR^d}\inf\setcond{\lambda\geq 0}{K\cap(x+L)\subset y+\lambda C}&\leq \inf_{\substack{L^\prime\in\LL^d_{j+1}\\L\subset L^\prime}} \inf_{y\in\RR^d}\inf\setcond{\lambda\geq 0}{K\cap(x+L^\prime)\subset y+\lambda C}\\
&\leq\sup_{\substack{L^\prime\in\LL^d_{j+1}\\L\subset L^\prime}}\inf_{y\in\RR^d}\inf\setcond{\lambda\geq 0}{K\cap(x+L^\prime)\subset y+\lambda C}.
\end{align*}
Consequently,
\begin{align*}
R_\sigma^j(K,C)&=\sup_{L\in\LL^d_{d-j}}\inf_{y\in\RR^d}\inf\setcond{\lambda\geq 0}{K\cap(x+L)\subset y+\lambda C}\\
&\leq \sup_{L\in\LL^d_j}\inf_{\substack{L^\prime\in\LL^d_{j+1}\\L\subset L^\prime}}\inf_{y\in\RR^d}\inf\setcond{\lambda\geq 0}{K\cap(x+L^\prime)\subset y+\lambda C}\\
&\leq\sup_{L\in\LL^d_j}\sup_{\substack{L^\prime\in\LL^d_{j+1}\\L\subset L^\prime}}\inf_{y\in\RR^d}\inf\setcond{\lambda\geq 0}{K\cap(x+L^\prime)\subset y+\lambda C}\\
&=R_\sigma^{j+1}(K,C)
\end{align*}
and
\begin{align*}
R^\sigma_j(K,C)&=\inf_{L\in\LL^d_j}\inf_{y\in\RR^d}\inf\setcond{\lambda\geq 0}{K\cap(x+L)\subset y+\lambda C}\\
&\leq \inf_{L\in\LL^d_j}\inf_{\substack{L^\prime\in\LL^d_{j+1}\\L\subset L^\prime}}\inf_{y\in\RR^d}\inf\setcond{\lambda\geq 0}{K\cap(x+L^\prime)\subset y+\lambda C}\\
&=R^\sigma_{j+1}(K,C).
\end{align*}
This completes the proof.
\end{proof}
Let us now state the monotonicity of the inradii counterparts. See again \cite{Gonzalez2013,BetkeHe1992} for the corresponding classical results, but also \cite{HenkHe2008b,HenkHe2009,GonzalezHeHi2015,GonzalezHe2012,GritzmannKl1992,Perelman1987}.
\begin{Satz}
We have $r_\pi^1(K,C)\geq\ldots\geq r_\pi^d$, $r^\pi_1(K,C)\geq\ldots\geq r^\pi_d(K,C)$, $r^\sigma_1\geq\ldots\geq\ r^\sigma_d$, and $r_\sigma^1(K,C)\geq \ldots\geq r_\sigma^d(K,C)$.
\end{Satz}
\begin{proof}
Observe that
\begin{align*}
r_\pi^j(K,C)&= \sup_{L\in \LL^d_{d-j}} r(K+L,C)=\sup_{L\in \LL^d_{d-j}} R(C,K+L)^{-1}\\
&=\lr{\inf_{L\in \LL^d_{d-j}} R(C,K+L)}^{-1}=R^\pi_j(C,K)^{-1},\\
r^\pi_j(K,C)&= \inf_{L\in \LL^d_{d-j}} r(K+L,C)=\inf_{L\in \LL^d_{d-j}} R(C,K+L)^{-1}\\
&=\lr{\sup_{L\in \LL^d_{d-j}} R(C,K+L)}^{-1}=R_\pi^j(C,K)^{-1}
\end{align*}
for all $j\in\setn{1,\ldots,d}$. This yields the first and the second claim. Now fix $j\in\setn{1,\ldots,d-1}$. Let $L\in \LL^d_j$ and $L^\prime\in\LL^d_{j+1}$ with $L\subset L^\prime$. For all $x,y\in\RR^d$ and all $\lambda\geq 0$, we have
\begin{align*}
y+\lambda((x+L)\cap C)\subset y+\lambda(C\cap (x+L^\prime)).
\end{align*}
Therefore, we obtain
\begin{equation*}
\inf_{x\in\RR^d}\sup_{y\in\RR^d}\sup\setcond{\lambda \geq 0}{y+\lambda(C\cap (x+L))}\geq \inf_{x\in\RR^d}\sup_{y\in\RR^d}\sup\setcond{\lambda \geq 0}{y+\lambda(C\cap (x+L^\prime))}.
\end{equation*}
It follows that
\begin{align*}
&\norel\inf_{x\in\RR^d}\sup_{y\in\RR^d}\sup\setcond{\lambda \geq 0}{y+\lambda(C\cap (x+L))}\\
&\geq\sup_{\substack{L^\prime\in\LL^d_{j+1}\\L\subset L^\prime}} \inf_{x\in\RR^d}\sup_{y\in\RR^d}\sup\setcond{\lambda \geq 0}{y+\lambda(C\cap (x+L^\prime))}\\
&\geq\inf_{\substack{L^\prime\in\LL^d_{j+1}\\L\subset L^\prime}} \inf_{x\in\RR^d}\sup_{y\in\RR^d}\sup\setcond{\lambda \geq 0}{y+\lambda(C\cap (x+L^\prime))}.
\end{align*}
Consequently,
\begin{align*}
r_\sigma^j(K,C)&=\sup_{L\in\LL^d_j}\inf_{x\in\RR^d}\sup_{y\in\RR^d}\sup\setcond{\lambda\geq 0}{y+\lambda (C\cap(x+L))\subset K}\\
&\geq \sup_{L\in\LL^d_j}\sup_{\substack{L^\prime\in\LL^d_{j+1}\\L\subset L^\prime}} \inf_{x\in\RR^d}\sup_{y\in\RR^d}\sup\setcond{\lambda \geq 0}{y+\lambda(C\cap (x+L^\prime))}\\
&=r_\sigma^{j+1}(K,C)
\end{align*}
and
\begin{align*}
r^\sigma_j(K,C)&=\inf_{L\in\LL^d_j}\inf_{x\in\RR^d}\sup_{y\in\RR^d}\sup\setcond{\lambda\geq 0}{y+\lambda (C\cap(x+L))\subset K}\\
&\geq \inf_{L\in\LL^d_j}\sup_{\substack{L^\prime\in\LL^d_{j+1}\\L\subset L^\prime}} \inf_{x\in\RR^d}\sup_{y\in\RR^d}\sup\setcond{\lambda \geq 0}{y+\lambda(C\cap (x+L^\prime))}\\
&\geq \inf_{L\in\LL^d_j}\inf_{\substack{L^\prime\in\LL^d_{j+1}\\L\subset L^\prime}} \inf_{x\in\RR^d}\sup_{y\in\RR^d}\sup\setcond{\lambda \geq 0}{y+\lambda(C\cap (x+L^\prime))}\\
&=r^\sigma_{j+1}(K,C),
\end{align*}
and the proof is complete.
\end{proof}
The next identities for the special cases $j\in\setn{1,d}$ follow directly from the definitions. Their classical counterparts are stated in every of the prementioned papers on successive radii, where the corresponding quantities appear.
\begin{Lem}
We have
\begin{align*}
R^\pi_d(K,C)=R_\pi^d(K,C)=R_\sigma^j(K,C)=R^\sigma_j(K,C)&=R(K,C),\\
r_\sigma^d(K,C)=r^\sigma_d(K,C)=r^\pi_d(K,C)=r_\pi^d(K,C)&=r(K,C),\\
R^\pi_1(K,C)=R^\sigma_1(K,C)=r^\pi_1(K,C)=r^\sigma_1(K,C)&=\frac{1}{2}\omega(K,C),\\
R_\pi^1(K,C)=R_\sigma^1(K,C)=r_\sigma^1(K,C)&=\frac{1}{2}D(K,C).
\end{align*}
\end{Lem}
In general, $r_\pi^1(K,C)=\sup\setcond{\frac{h_{K-K}(u)}{h_{C-C}(u)}}{u\in\RR^d\setminus\setn{0}}\neq \frac{1}{2}D(K,C)$, see \cite[Example~3.14(a)]{Jahn2015b}.

Finally, successive radii inherit the invariance under hull operations and translations as well as monotonicity with respect to set inclusion and compatibility with scaling.
\begin{Lem}
Let $K,K^\prime\subset \RR^d$, $C,C^\prime\in\KK^d_0$, and $\alpha,\beta>0$. If $f(K,C)$ is one of the quantities defined in Definitions~\ref{def:inner_outer} and \ref{def:successive_radii}, then the following identities are true:
\begin{enumerate}[label={(\alph*)},align=left]
\item{$f(K^\prime,C^\prime)\leq f(K,C)$ if $K^\prime \subset K$ and $C\subset C^\prime$,}
\item{$f(K,C)=f(\cl(K),C)=f(\co(K),C)$,}
\item{$f(x+K,y+C)=f(K,C)$ for all $x,y\in\RR^d$,}
\item{$f(\alpha K,\beta C)=\frac{\alpha}{\beta}f(K,C)$.}
\end{enumerate}
\end{Lem}
\pagebreak
\section{Intersections of balls of a specific radius}\label{chap:translates}
In Minkowski geometry, \dah in the geometry of finite-dimensional real Banach spaces, there are two hull notions which occur as special intersections of balls. Namely, the \emph{ball hull} of a given set is the intersection of all balls that contain this set and have a specific radius, whereas the \emph{ball intersection} of a given set is the intersection of all balls with a specific radius and centers in this set (see \cite[Section~2]{Spirova2010a} and \cite{MartinMaSp2014}). Alternatively, the ball intersection of a given set can be seen as the set of centers of balls that have a specific radius and contain this set. This yields a link to minimax location problems (which are also known as center problems).
The level sets of these problems are ball intersections, and the optimal value is the circumradius. In our general setting, we replace the unit ball by an arbitrary convex set $C$ which need not to be centered. Therefore it is suitable to modify the definition of the ball intersection by taking the intersection of scaled copies of $-C$, not of $C$.
\begin{Def}\label{def:intersections}
Let $K\subset \RR^d$ be bounded, $C\in \KK^d_0$ with $0\in\inte(C)$, and let $\lambda\geq R(K,C)$.
The \emph{ball intersection} of $K$ (with respect to $C$) is defined as
\begin{equation*}
\bi(K,C,\lambda)\defeq \bigcap_{x\in K}(x-\lambda C).
\end{equation*}
The \emph{ball hull} (with respect to $C$) is defined as
\begin{equation*}
\bh(K,C,\lambda)\defeq \bigcap_{x\in\RR^d:\,K\subset x+\lambda C}(x+\lambda C).
\end{equation*}
\end{Def}
The restriction $\lambda\geq R(K,C)$ is a consequence of the observation 
\begin{align}
R(K,C)&=\inf\setcond{\lambda\geq 0}{\bigcap_{y\in K} (y-\lambda C)\neq \emptyset}\label{eq:circumradius_via_ball_intersection}
\end{align}
and the standard compactness argument for showing the existence of an optimal solution of the convex optimization problem 
\begin{equation}
\inf_{x\in\RR^d}\sup_{y\in K}\gamma_C(y-x),\label{eq:minimax}
\end{equation}
whose optimal value is $R(K,C)$; see again \eqref{eq:inclusion_vs_optimization}. Figure~\ref{fig:ball_hull} shows two examples of ball hulls for polygonal sets $K$ and $C$.

\begin{figure}[h!]
\begin{center}
\begin{tikzpicture}[line cap=round,line join=round,>=triangle 45,x=0.8cm,y=0.8cm]
\draw[color=black,shift={(0.6,1.7)}] (0,2)--(-1.9,0.62)--(-1.18,-1.62)--(1.18,-1.62)--(1.9,0.62)--cycle;
\draw[color=black] (3.4,3)--(3.9,2.11)--(5.46,2.59)--cycle;
\draw[color=black,dashed] (4.5,3.98)--(3.35,3.15)--(3.79,1.81)--(5.2,1.81)--(5.64,3.15)--cycle;
\draw[color=black,dashed] (4.42,3.74)--(3.28,2.91)--(3.71,1.57)--(5.12,1.57)--(5.56,2.91)--cycle;
\draw[color=black,line width=2pt] (3.79,1.81)--(5.2,1.81)--(5.56,2.91)--(4.42,3.74)--(3.4,3)--cycle;
\draw[shift={(0.6,0.7)}] (0,0) node{$C$};
\draw (4.1,2.5) node{$K$};
\draw[color=black,shift={(4.5,0)}] (3.4,3)--(3.9,2.11)--(5.46,2.59)--cycle;
\draw[color=black,dashed,shift={(4.5,0)}] (4.12,5.73)--(2.22,4.35)--(2.95,2.11)--(5.3,2.11)--(6.03,4.35)--cycle;
\draw[color=black,dashed,shift={(4.5,0)}] (4.87,5.73)--(2.97,4.35)--(3.69,2.11)--(6.04,2.11)--(6.77,4.35)--cycle;
\draw[color=black,dashed,shift={(4.5,0)}] (4.15,3.54)--(2.25,2.16)--(2.97,-0.08)--(5.32,-0.08)--(6.05,2.16)--cycle;
\draw[color=black,line width=2pt,shift={(4.5,0)}] (3.69,2.11)--(5.3,2.11)--(5.46,2.59)--(4.15,3.54)--(3.4,3)--cycle;
\draw[shift={(4.5,0)}] (4.1,2.5) node{$K^\prime$};
\end{tikzpicture}
\end{center}\caption{The ball hull $\bh(K,C,0.6)$ and $\bh(K^\prime,C,1)$ are shown in bold lines, where $K$ and $K^\prime$ are triangles, and $C$ is a regular pentagon.}\label{fig:ball_hull}
\end{figure}
\pagebreak
When combining the notions of ball hull and ball intersection, one obtains almost the same formulas as in normed spaces.
\begin{Satz}[{\cite[Proposition~2.1.1,~2.1.2]{Spirova2010a}}]
Let $K,K^\prime\subset\RR^d$, and $C\in\KK^d_0$ such that $0\in\inte(C)$.
\begin{enumerate}[label={(\alph*)},align=left]
\item{Let $\lambda\geq R(K^\prime,C)$ and $K\subset K^\prime$. We have $\bi(K,C,\lambda)\supset\bi(K^\prime,C,\lambda)$ and $\bh(K,C,\lambda)\subset \bh(K^\prime,C,\lambda)$.\label{intersections_set_monotone}}
\item{If $R(K,C)\leq\lambda\leq \lambda^\prime$, then $\bi(K,C,\lambda)\subset\bi(K,C,\lambda^\prime)$ and $\bh(K,C,\lambda)\supset \bh(K,C,\lambda^\prime)$.\label{intersections_radius_monotone}}
\item{If $\bh(K,C,\lambda)\neq\emptyset$ and $\bi(K,C,\lambda)\neq\emptyset$, then $\bh(K,C,\lambda)=\bi(\bi(K,C,\lambda),-C,\lambda)$ and $\bi(K,C,\lambda)=\bi(\bh(K,C,\lambda),C,\lambda)$.\label{intersections_compositions}}
\item{If $\sup\setcond{\gamma_C(x-y)}{x,y\in K}=\lambda$, then $\bh(K,C,\lambda)\subset \bi(K,-C,\lambda)$.}
\end{enumerate}
\end{Satz}
\begin{proof}
Statement~\ref{intersections_set_monotone} follows directly from Definition~\ref{def:intersections}. In order to prove the first statement of \ref{intersections_radius_monotone}, observe that $x-\lambda C\subset x-\lambda^\prime C$ for all $x\in K$. This implies
\begin{equation*}
\bi(K,C,\lambda)=\bigcap_{x\in K}(x-\lambda C)\subset\bigcap_{x\in K}(x-\lambda^\prime C)=\bi(K,C,\lambda^\prime).
\end{equation*}
The second part of \ref{intersections_radius_monotone} is easy as well. Observe that $K\subset x+\lambda^\prime C$ if $K\subset x+\lambda C$. Therefore $\setcond{x\in\RR^d}{x+\lambda C\supset K}\subset \setcond{x\in\RR^d}{x+\lambda^\prime C\supset K}$, which yields $\bh(K,C,\lambda)\supset \bh(K,C,\lambda^\prime)$.
Let us prove the first statement in \ref{intersections_compositions}. We have
\begin{align*}
&\noequiv y\in \bh(K,C,\lambda)\\
&\Longleftrightarrow y\in \bigcap_{x\in\RR^d:\,K\subset x+\lambda C}(x+\lambda C)\\
&\Longleftrightarrow y\in \bigcap_{x \in \bi(K,C,\lambda)}(x+\lambda C)\\
&\Longleftrightarrow y\in \bi(\bi(K,C,\lambda),-C,\lambda),
\end{align*}
simply because, for $x\in\RR^d$, we have $K\subset x+\lambda C \Longleftrightarrow z\in x+\lambda C \fall z\in K\Longleftrightarrow x\in z-\lambda C \fall z\in K\Longleftrightarrow x\in\bi(K,C,\lambda)$. The second statement in \ref{intersections_compositions} follows from the representation
\begin{equation*}
\bi(K,C,\lambda)=\setcond{x\in\RR^d}{K\subset x+\lambda C}
\end{equation*}
and the equivalence $K\subset x+\lambda C \Longleftrightarrow \bh(K,C,\lambda)\subset x+\lambda C$. Indeed, if $K$ is contained in a translate of $\lambda C$, then $\bh(K,C,\lambda)$ is a subset of this translate because it is the intersection of all translates of $\lambda C$ that contain $K$. Conversely, since $\bh(K,C,\lambda)$ is the intersection of all translates of $\lambda C$ that contain $K$, it follows that $K$ is a subset of $\bh(K,C,\lambda)$ itself, and therefore it is contained in each translate of $\lambda C$ that contains $\bh(K,C,\lambda)$.

Now assume that $\sup\setcond{\gamma_C(x-y)}{x,y\in K}=\lambda$. Then, for all $x\in K$, the set $x+\lambda C$ is a translate of $\lambda C$ that contains $K$. Hence $\bh(K,C,\lambda)\subset \bi(K,-C,\lambda)$.
\end{proof}
\begin{Bem}
Since we have $K\subset x+\lambda C \Longleftrightarrow \bh(K,C,\lambda)\subset x+\lambda C$, it follows that the relation $\bh(K,C,\lambda)=\bh(\bh(K,C,\lambda),C,\lambda)$ is valid. Let us also prove the relation
\begin{equation*}
\bi(K,C,\lambda)=\bh(\bi(K,C,\lambda),-C,\lambda).
\end{equation*}
The inclusion \enquote{$\subset$} is trivial because $\bh(J,-C,\lambda)$ is, by definition, the intersection of supersets of $J\defeq \bi(K,C,\lambda)$. It remains to show the reverse inclusion. By definition, $K\subset\setcond{y\in\RR^d}{y-\lambda C\supset\bi(K,C,\lambda)}$. We obtain
\begin{align*}
\bi(K,C,\lambda)&=\bigcap_{y\in K} (y-\lambda C)\\
&\supset \bigcap_{y\in\RR^d:\,y-\lambda C\supset\bi(K,C,\lambda)} (y-\lambda C)\\
&=\bh(\bi(K,C,\lambda),-C,\lambda).
\end{align*}
\end{Bem}

A special case of the ball intersection is the set 
\begin{equation*}
\cc(K,C)\defeq \setcond{x\in\RR^d}{K\subset x+R(K,C)C}
\end{equation*}
of \emph{circumcenters} of $K$ with respect to $C$. (Sometimes, it is also called the \emph{Chebyshev set} of $K$, see \cite{MartinMaSp2014}.) This follows easily from \eqref{eq:circumradius_via_ball_intersection}. The set $\cc(K,C)$ can also be viewed as the solution set of the optimization problem \eqref{eq:minimax}.
The set 
\begin{equation*}
\ic(K,C)\defeq \setcond{x\in\RR^d}{x+r(K,C)C\subset K}
\end{equation*}
is called the set of \emph{incenters} of $K$ with respect to $C$. A short computation shows that introducing the inradius in a separate way is not necessary because it is the inverse of the circumradius for interchanged arguments:
\begin{align}
R(C,K)&=\inf\setcond{\lambda\geq 0}{\ex x\in \RR^d: C\subset x+\lambda K}\nonumber\\
&=\inf\setcond{\lambda^{-1}\geq 0}{\ex x\in \RR^d: C\subset x+\lambda K}^{-1}\nonumber\\
&=\inf\setcond{\lambda^{-1}\geq 0}{\ex x\in \RR^d: -\lambda^{-1}x+\lambda^{-1}C\subset K}^{-1}\nonumber\\
&=\inf\setcond{\mu\geq 0}{\ex x\in \RR^d: y+\mu C\subset K}^{-1}\nonumber\\
&=r(K,C)^{-1},\label{eq:inradius_vs_circumradius}
\end{align}
see \cite[p.~5]{BrandenbergKoe2014}. Having a closer look at this computation, we see that the set $\ic(K,C)$ can be computed in terms of $\cc(C,K)$ and $r(K,C)$. Namely, $y\in \ic(K,C)$ if and only if there exists $x\in \cc(C,K)$ such that $y=-r(K,C)x$. In other words,
\begin{equation}
\ic(K,C)=-r(K,C)\cc(C,K).\label{eq:ic_vs_cc}
\end{equation}
Now assume that $K,C\in \KK^d_0$. By the above reasons, $K$ has a circumcenter with respect to $C$ and vice versa. Furthermore, one can easily verify that $R(K,C),R(C,K)>0$. By virtue of \eqref{eq:inradius_vs_circumradius} and \eqref{eq:ic_vs_cc}, this is sufficient for $r(K,C)>0$ and $\ic(K,C)\neq\emptyset$. Equation \eqref{eq:ic_vs_cc} also tells us that, for understanding circumcenter sets and incenter sets, it suffices to describe the shape of $\cc(K,C)$, where $K$ and $C$ are arbitrary sets from $\KK^d_0$. 
Hence, $\cc(K,C)=\bigcap_{y\in K} (y-R(K,C) C)=\bi(K,C,R(K,C))$ is the intersection of a family of bounded closed convex sets.
\begin{Prop}
For any $K,C\in \KK^d_0$, the set $\cc(K,C)$ is non-empty, bounded, closed, and convex. Moreover, if $C$ is a polytope, then so is $\cc(K,C)$. 
\end{Prop}
\begin{Kor}
For any $K,C\in \KK^d_0$, the set $\ic(K,C)$ is non-empty, bounded, closed, and convex. Moreover, if $K$ is a polytope, then so is $\ic(K,C)$. 
\end{Kor}
A direct consequence of the definition and formula \eqref{eq:ic_vs_cc} is the following statement.
\begin{Prop}
If $K,C\in\KK^d_0$ are centrally symmetric, then so are $\cc(K,C)$ and $\ic(K,C)$.
\end{Prop}
To our best knowledge, the first result concerning the dimension of the set of incenters is \cite[Satz~5]{Zindler1920} which only concerns the Euclidean plane. In the original version of Bonnesen and Fenchel's monograph from 1934 (see \cite[p.~59]{BonnesenFe1987} in the 1987 edition), one finds the following statement without any proof. The set of incenters of an arbitrary convex body in $d$ dimensions is an arbitrary convex body of dimension at most $d-1$. We will prove this property (of being not full-dimensional) for the set of circumcenters. Then the statement of Bonnesen and Fenchel follows immediately by \eqref{eq:ic_vs_cc}.
\begin{Satz}
For any $K,C\in \KK^d_0$, we have $\dim(\cc(K,C))\leq d-1$.
\end{Satz}
\begin{proof}
By convexity of $\cc(K,C)$, it suffices to prove the emptiness of its interior. By the invariance of $R(K,C)$ under translations of its arguments, we may assume $0\in\inte(C)$. Assume there is a point $x\in\inte(\cc(K,C))$. Then there exist $\eps>0$ and $y\in K$ with $x+\eps C\in \cc(K,C)$ and $\gamma_C(y-x)=R(K,C)$. On the one hand we have 
\begin{equation*}
z\defeq y-\lr{1+\frac{\eps}{2\gamma_C(y-x)}}(y-x)=x+\frac{\eps}{2\gamma_Cy-x)}(y-x)\in x+\eps C\subset \cc(K,C),
\end{equation*}
but on the other hand $\gamma_C(y-z)=\bigl(1+\frac{\eps}{2\gamma_C(y-x)}\bigr)\gamma_C(y-x)>\gamma_C(y-x)=R(K,C)$. This implies $y\notin z+R(K,C)C$, a contradiction.
\end{proof}

\section{Concluding remarks}\label{chap:conclusions}
In the present paper, we provide some results that can be taken as starting points for advanced research on metrical problems in generalized Minkowski spaces. However, there are some questions left.

It is convenient to alter the definitions of those classical successive radii which incoporate projections such that they rather use cylinders. What happens if we still use the classical definitions but alter the notion of projections? More precisely, replace $\Proj_A(x)$ by $\Proj^C_A(x)\defeq \setcond{y\in\RR^d}{\gamma_C(x-y)=\dist_C(x,A)}$ where $A$ is an affine flat. The question is whether basic properties, like the monotonicity with respect to the dimension of the participating affine flats, still remain valid?

Can we estabish estimates for taking Minkowski sums in the first argument of successive radii like in \cite{Gonzalez2013}? What are appropriate definitions of successive diameters and width in generalized Minkowski spaces, and what can be said about continuity of all these quantities (cf. \cite{BetkeHe1992})? Is there another way for defining the quantities $r_\pi^j$ in order to have also $r_\pi^1=\frac{1}{2}D$? 

Kolmogorov numbers of a linear operator $T:X\to Y$ between Banach spaces $X$ and $Y$ are well-studied quantities in approximation theory, see \cite[p.~95]{Pietsch1987}. In the setting of normed spaces, the $j$th Kolmogorov number is commonly defined as 
\begin{equation*}
\inf\setcond{\mnorm{Q_N\circ T}}{N \text{ linear subspace of }Y,\dim(N)<j},
\end{equation*}
where $Q_N$ is the quotient mapping $Y\to Y/N$ and $\mnorm{Q_N\circ T}$ is the usual operator norm on the vector space of bounded linear operators $X\to Y/N$. As the authors of \cite{GonzalezHeHi2015} point out, the $j$th Kolmogorov number is equal to
\begin{equation*}
\inf_{\substack{N \text{ linear subspace of }Y,\\\dim(N)<j}}\sup_{\mnorm{x}\leq 1}\inf_{y\in N}\mnorm{T(x)-y}.
\end{equation*}
Using the monotonicity of this quantity with respect to $j$ and \cite[Proposition~(3.1)]{AlegreFe2007}, we can translate Kolmogorov numbers into our setting. Given $K,C\in \KK^d_0$ with $0\in\inte{C}\cap \inte(K)$, consider
\begin{equation*}
\inf_{N\in \LL^d_{j-1}}\sup_{x\in K}\inf_{y\in N}\gamma_C(T(x)-y).
\end{equation*}
For $T:\RR^d\to\RR^d$, $T(x)=x$, this number is simply $R^\pi_{d-j+1}(K,C)$. Is there an appropriate generalization of other so-called \emph{$s$-numbers} to gauges, which can be nicely combined with $r_\sigma^j$ and $R^\pi_j$ like in \cite[Theorem~2.1]{GonzalezHeHi2015}?

\providecommand{\bysame}{\leavevmode\hbox to3em{\hrulefill}\thinspace}

\end{document}